\documentclass[11pt,a4paper]{amsart}

\usepackage{amssymb,mathrsfs,enumitem,amsmath,amsthm,bbold}
\usepackage[mathscr]{euscript}
\usepackage[all]{xy}
\usepackage[in]{fullpage}
\usepackage{color,pxfonts}
\definecolor{Chocolat}{rgb}{0.36, 0.2, 0.09}
\definecolor{BleuTresFonce}{rgb}{0.215, 0.215, 0.36}
\definecolor{EgyptianBlue}{rgb}{0.06, 0.2, 0.65}

\usepackage[colorlinks,final,hyperindex]{hyperref}
\hypersetup{citecolor=BleuTresFonce, urlcolor=EgyptianBlue, linkcolor=Chocolat}

\usepackage{fourier}

\catcode`\@=11

\catcode`\@=13

\newtheorem{theorem}{Theorem}

\newtheorem{proposition}{Proposition}

\theoremstyle{definition}

\DeclareMathAlphabet{\pazocal}{OMS}{zplm}{m}{n}

\def\calP{\pazocal{P}}
\def\calQ{\pazocal{Q}}

\def\calX{\pazocal{X}}

\DeclareMathAlphabet{\mathbbold}{U}{bbold}{m}{n}

\def\kk{\mathbbold{k}}

\DeclareMathOperator{\gr}{gr}

\DeclareMathOperator{\RBLie}{\textsl{RBLie}_1}
\DeclareMathOperator{\RBAss}{\textsl{RBAss}_1}
\DeclareMathOperator{\PostLie}{\textsl{PostLie}}
\DeclareMathOperator{\ComTriAss}{\textsl{ComTriAss}}
\DeclareMathOperator{\TriDend}{\textsl{TriDend}}
\DeclareMathOperator{\PostPoisson}{\textsl{PostPoisson}}

\begin{document}

\title{Functorial PBW theorems for post-Lie algebras}

\author{Vladimir Dotsenko}
\address{School of Mathematics, Trinity College, Dublin 2, Ireland}
\email{vdots@maths.tcd.ie}


\begin{abstract}
Using the categorical approach to Poincar\'e--Birkhoff--Witt type theorems from our previous work with Tamaroff, we prove three such theorems: for universal enveloping Rota--Baxter algebras of tridendriform algebras, for universal enveloping Rota--Baxter Lie algebras of post-Lie algebras, and for universal enveloping tridendriform algebras of post-Lie algebras. Similar results, though without functoriality of the PBW isomorphisms, were recently obtained by Gubarev. Our methods are completely different and mainly rely on methods of rewriting theory for shuffle operads.
\end{abstract}

\maketitle

\section{Introduction}

A vector space is said to have a structure of a (right) post-Lie algebra if it is equipped with an anti-commutative binary operation $a_1,a_2\mapsto [a_1,a_2]$ and a binary operation $a_1,a_2\mapsto a_1\triangleleft a_2$ without any symmetries that satisfy the identities
\begin{gather*}
[[a_1,a_2],a_3]-[[a_1,a_3],a_2]=[a_1,[a_2,a_3]],\\
[a_1,a_2]\triangleleft a_3=[a_1\triangleleft a_3,a_2]+[a_1,a_2\triangleleft a_3],\\
\left((a_1\triangleleft a_2)\triangleleft a_3-a_1\triangleleft(a_2\triangleleft a_3)\right)-\left((a_1\triangleleft a_3)\triangleleft a_2-a_1\triangleleft(a_3\triangleleft a_2)\right)=a_1\triangleleft[a_2,a_3].
\end{gather*}
This algebraic structure was introduced by Vallette~\cite{Val} as an interesting example illustrating some methods of operad theory, although it turned out to have had been known implicitly long before that. Perhaps the most natural geometric example of such structure arises on the tangent bundle of a manifold equipped with a flat connection of constant torsion; there are other exciting instances of post-Lie algebras in a plethora of research areas including Lie theory (when studying affine actions of Lie groups and generalised derivations of semisimple Lie algebras~\cite{BDV,BD}), numerical integration of differential equations (the Lie--Butcher formalism~\cite{LM}), operad theory (functorial splittings of algebraic operations via Manin products for operads~\cite{BBGN,VM}), and even in some aspects of the classical Yang--Baxter equation~\cite{BGN}.

There are two different ways to assign meaningful universal enveloping algebras to post-Lie algebras. In the context of Yang--Baxter equations, a natural universal enveloping algebra would have a structure of the so called Rota--Baxter Lie algebra of weight one, an algebraic notion closely related to the classical Yang--Baxter equation~\cite{BeDr,Sem}. Operad theory suggests a yet another good algebraic structure that would provide a meaningful definition of a universal enveloping algebra, tridendriform algebras of Loday and Ronco~\cite{LR}; relevance of those algebras are suggested by the recent application of tridendriform algebras to the Magnus expansion \cite{EFM}. In each of those situations, it is natural to ask whether there is a Poincar\'e--Birkhoff--Witt (PBW) type theorem describing the underlying vector space of the universal enveloping algebra
(putting this into the framework of PBW pairs of varieties of algebras~\cite{MiSh}). Recently, this question was studied by Gubarev who proved appropriate PBW type theorems in \cite{Gu1,Gu2}. 

The goal of this paper is to offer an alternative approach to Gubarev's theorems relying on the categorical framework for PBW theorems developed in a joint work with Tamaroff~\cite{DT}. Besides being an illustration of concrete methods of operad theory \cite{BrDo} at work and benefiting the reader who prefers the language of operad theory where possible, our PBW theorems are canonical, i.e. we prove that the PBW isomorphism can be chosen functorially with respect to algebra morphisms, which is never apparent if one chooses to use methods of \cite{Gu1,Gu2}.

It is perhaps worth mentioning that we do not consider in this paper another important kind of universal enveloping algebras of post-Lie algebras, the so called D-algebras prominent in the Lie--Butcher calculus~\cite{EFLM,LM}. Those algebras are obtained via an adjunction not arising from change of algebraic structure and, as a consequence, are a bit different; a PBW type theorem for them is true on the nose because of the Lie-theoretic nature of the definition. 

\subsection*{Acknowledgements. } I am grateful to Vsevolod Gubarev for a discussion of results of \cite{Gu2}, and to Murray Bremner for his comments on the paper \cite{DT}, and particularly for a query as to how results of that paper may be applied to post-Lie algebras.

\section{Recollections}

This is a short note, and we do not intend to overload it with excessive recollections. For relevant information on symmetric operads and Koszul duality, we refer the reader to the monograph \cite{LV}, and for information on shuffle operads, Gr\"obner bases and rewriting systems to the monograph \cite{BrDo}. We take the liberty to say, for a symmetric operad $\calP$, ``the shuffle operad $\calP$'' where one should really say ``the shuffle operad obtained from $\calP$ by applying the forgetful functor''. 

All operads in this paper are defined over a field $\kk$ of characteristic zero. We assume all operads reduced ($\mathcal{P}(0)=0$) and connected ($\mathcal{P}(1)=\kk$). When writing down elements of operads, we use arguments $a_1$, \ldots, $a_n$ as placeholders; any nontrivial signs (in case one wishes to, say, work with differential graded post-Lie algebras) would only arise from applying operations to arguments via the usual Koszul sign rule.

\subsection*{Tridendriform algebras}

In \cite{Val}, it is implicitly indicated that the operad $\PostLie$ of post-Lie algebras is related to operad $\TriDend$ of so called tridendriform algebras, or dendriform trialgebras studied by Loday and Ronco~\cite{LR}. A tridendriform algebra has three binary operations $a_1,a_2\mapsto a_1\prec a_2$, $a_1,a_2\mapsto a_1\succ a_2$, and $a_1,a_2\mapsto a_1\cdot a_2$ without any symmetries that satisfy the identities
\begin{gather*}
(a_1\prec a_2)\prec a_3=a_1\prec(a_2\prec a_3)+a_1\prec(a_2\succ a_3)+a_1\prec(a_2\cdot a_3),\\
(a_1\cdot a_2)\prec a_3=a_1\cdot(a_2\prec a_3),\qquad
(a_1\cdot a_2)\cdot a_3=a_1\cdot(a_2\cdot a_3),\\
(a_1\succ a_2)\prec a_3=a_1\succ(a_2\prec a_3),\\
(a_1\prec a_2)\cdot a_3=a_1\cdot(a_2\succ a_3),\qquad
(a_1\succ a_2)\cdot a_3=a_1\succ(a_2\cdot a_3),\\
(a_1\prec a_2)\succ a_3+(a_1\succ a_2)\succ a_3+(a_1\cdot a_2)\succ a_3=a_1\succ(a_2\succ a_3).
\end{gather*}
A precise relationship between tridendriform algebras and post-Lie algebras was described in \cite[Proposition 5.13]{BGN}: any tridendriform algebra can be made into a post-Lie algebra by considering the operations $[a_1,a_2]:=a_1\cdot a_2-a_2\cdot a_1$ and $a_1\triangleleft a_2:=a_1\prec a_2-a_2\succ a_1$. 

\subsection*{Rota--Baxter algebras}

From the general formalism of splitting of operations in operads, it is possible to relate post-Lie algebras and tridendriform algebras to Rota--Baxter algebras~\cite[Theorem 5.4]{BBGN} (for tridendriform algebras, this was observed earlier in~\cite[Section 4]{EF}). Let us recall the relevant definitions and statements. 

The operad $\RBLie$ of Rota--Baxter Lie algebras of weight one is generated by an anti-commutative binary operation $a_1,a_2\mapsto [a_1,a_2]$ and a unary operation $R$ that satisfy the identities
\begin{gather}
[[a_1,a_2],a_3]-[[a_1,a_3],a_2]=[a_1,[a_2,a_3]],\\
[R(a_1),R(a_2)]=R([R(a_1),a_2]+[a_1,R(a_2)]+[a_1,a_2]).\label{eq:RBL}
\end{gather}
Every Rota--Baxter Lie algebra of weight one can be made into a post-Lie algebra by considering the operation $a_1\triangleleft a_2:=[a_1,R(a_2)]$.

Let us remark that in the literature, one can instead find the definition of Rota-Baxter Lie algebras of weight $\lambda\ne 0$ where the second identity becomes $[R(a_1),R(a_2)]=R([R(a_1),a_2]+[a_1,R(a_2)]+\lambda[a_1,a_2])$. Note that replacing the operator $R$ in this latter identity by $-\lambda R$ implements an isomorphism between the corresponding operads for all possible nonzero values of $\lambda$. A particularly notable choice $\lambda=-1$ appears in the original definition of this algebraic structure, going back to the seminal article of Semenov-Tyan-Shanskii~\cite{Sem}.

The operad $\RBAss$ of Rota--Baxter associative algebras of weight one is generated by a binary operation $a_1,a_2\mapsto a_1\cdot a_2$ without any symmetries and a unary operation $R$ that satisfy the identities
\begin{gather}
(a_1\cdot a_2)\cdot a_3=a_1\cdot(a_2\cdot a_3),\\
R(a_1)\cdot R(a_2)=R(R(a_1)\cdot a_2+a_1\cdot R(a_2)+a_1\cdot a_2).\label{eq:RB}
\end{gather}
Every Rota--Baxter associative algebra of weight one can be made into a post-Lie algebra by considering the operations $a_1\prec a_2:=a_1\cdot R(a_2)$ and $a_1\succ a_2:=R(a_1)\cdot a_2$.

\subsection*{Functorial PBW theorems}

In joint work with Tamaroff \cite{DT}, we developed a categorical framework for PBW type theorems. Let us recall a slightly simplified version of the main result of that paper. Suppose that $\phi\colon\calP\rightarrow\calQ$ is a morphism of operads. It leads to a natural functor $\phi^*$ from the category of $\calQ$-algebras to the category of $\calP$-algebras (pullback of the structure). This functor admits a left adjoint $\phi_!$ computed via the relative composite product formula $\phi_!(A)=\calQ\circ_\calP A$, where $A$ in the latter formula is regarded as a ``constant analytic endofunctor'' (a symmetric sequence supported at arity zero). 
We say that the datum $(\calP,\calQ,\phi)$ \emph{has the PBW property} if there exists an endofunctor~$\calX$ such that the underlying object of the universal enveloping $\calQ$-algebra $\phi_!(A)$ of any $\calP$-algebra~$A$ is isomorphic to~$\calX(A)$ naturally with respect to $\calP$-algebra morphisms. 

\begin{proposition}[{\cite[Theorem 1]{DT}}]\label{th:PBWNat}
Let  $\phi\colon\calP\rightarrow\calQ$ be a morphism of operads. The datum $(\calP,\calQ,\phi)$ has the PBW property if and only if the right $\calP$-module action on $\calQ$ via $\phi$ is free.
\end{proposition}

The two constructions of post-Lie algebras mentioned above actually fit into a commutative diagram of operads
 \[
 \xymatrix{
\PostLie\ar@{->}^{\phi}[rr] \ar@{->}_{\alpha}[d] & & \ \TriDend \ar@{->}^{\beta}[d] \\ 
 \RBLie\ar@{->}^{\psi}[rr]  & &\RBAss      
 }
 \]
where the morphisms are defined by 
\begin{gather*}
\alpha([a_1,a_2])=[a_1,a_2],\qquad \alpha(a_1\triangleleft a_2)=[a_1,R(a_2)],\\
\beta(a_1\cdot a_2)=a_1\cdot a_2,\qquad \beta(a_1\prec a_2)=a_1\cdot R(a_2), \qquad \beta(a_1\succ a_2)=R(a_1)\cdot a_2,\\
\phi([a_1,a_2])=a_1\cdot a_2-a_2\cdot a_1,\qquad \phi(a_1\triangleleft a_2)=a_1\prec a_2-a_2\succ a_1,\\
\psi([a_1,a_2])=a_1\cdot a_2-a_2\cdot a_1,\qquad \psi(R(a_1))=R(a_1).
\end{gather*}

Some Poincar\'e--Birkhoff-Witt type results for algebras over operads involved in this commutative diagram were established by Gubarev who proved, using Gr\"obner--Shirshov bases in Rota-Baxter algebras, that a PBW type theorem holds for universal enveloping $\RBAss$-algebras of $\TriDend$-algebras \cite{Gu3}, for universal enveloping $\RBLie$-algebras of $\PostLie$-algebras \cite{Gu1}, and, more recently, for universal enveloping $\TriDend$-algebras of $\PostLie$-algebras \cite{Gu2}. He also established (private communication) that a PBW-type theorem does not hold for universal enveloping $\RBAss$-algebras of $\RBLie$-algebras. 

While Gr\"obner--Shirshov methods instantly lead to normal forms in universal enveloping algebras and as such are of course useful for applications, it would be desirable to have, for results about objects defined categorically via a universal property, canonical proofs that do not rely on arbitrary choices (such as a choices of ordered bases in algebras). It turns out that our categorical approach to the PBW property is applicable in all those cases. Below, we prove functorial versions of Gubarev's results using concrete algorithmic methods of operad theory. 

\section{PBW theorem for universal enveloping Rota-Baxter algebras of tridendriform algebras}

In this section, we prove a ``toy example'', the PBW theorem for universal enveloping $\RBAss$-algebras of $\TriDend$-algebras. This case is substantially simpler since all operads involved are obtained by symmetrization of nonsymmetric operads, and so one can perform all the operadic computation in the nonsymmetric universe where computational complexity is usually much lower. 

\begin{theorem}\label{th:TDRB}
The datum $(\TriDend,\RBAss,\beta)$ has the PBW property. 
\end{theorem}

\begin{proof}
It is sufficient to prove the free right module property in the universe of nonsymmetric operads, since symmetrisation takes nonsymmetric compositions to symmetric compositions. To not complicate notation, we identify operations with their images under $\beta$, so we use the name $a_1\prec a_2$ for the operation $a_1\cdot R(a_2)\in \RBAss$, and the name $a_1\succ a_2$ for the operation  $R(a_1)\cdot a_2\in \RBAss$. 

In this case, it is beneficial to consider rewriting systems rather than Gr\"obner bases. Let us show that the operad $\RBAss$ can be presented by the following rewriting system:
\begin{gather}
R(a_1)\cdot a_2\to a_1\succ a_2,\label{eq:TDelimR}\\
a_1\cdot R(a_2)\to a_1\prec a_2,\label{eq:TDelimR'}\\
a_1\succ R(a_2)\to R(a_1\succ a_2+a_1\prec a_2+a_1\cdot a_2),\label{eq:TD-a-R}\\
R(a_1)\prec a_2\to R(a_1\succ a_2+a_1\prec a_2+a_1\cdot a_2),\label{eq:TD-R-b}\\
a_1\prec (a_2\prec a_3)\to -a_1\prec(a_2\succ a_3)-a_1\prec(a_2\cdot a_3)+(a_1\prec a_2)\prec a_3,\\
a_1\cdot (a_2\prec a_3)\to (a_1\cdot a_2)\prec a_3,\\
a_1\cdot(a_2\cdot a_3)\to (a_1\cdot a_2)\cdot a_3,\\
a_1\succ (a_2\prec a_3)\to (a_1\succ a_2)\prec a_3,\\
a_1\cdot (a_2\succ a_3)\to (a_1\prec a_2)\cdot a_3,\label{eq:rmod}\\
a_1\succ (a_2\cdot a_3)\to (a_1\succ a_2)\cdot a_3,\\
a_1\succ (a_2\succ a_3)\to (a_1\succ a_2)\succ a_3+(a_1\prec a_2)\succ a_3+(a_1\cdot a_2)\succ a_3.
\end{gather}
We start with a remark that will be useful in the proof. Let us first focus on the rewriting rules that do not involve $R$ directly. Those are defining relations of the tridendriform operad which are known to hold in $\RBAss$. It turns out that if we impose the reverse length path-lexicographic order with $(a_1\cdot a_2)<(a_1\succ a_2)<(a_1\prec a_2)$, these rewriting rules single out the leading terms for this admissible ordering. This means that we have an upper bound on dimensions of the Koszul dual operad $\TriDend^!$: such an upper bound is given by the number of tree monomials whose divisors are the leading terms of the relations. By an easy computation, the upper bound in arity $n$ thus obtained is $2^n-1$ which coincides with the known dimension formula for $\dim\TriDend^!(n)$ \cite{LR}. Since the bound that we obtained is sharp, this means that our relations form a Gr\"obner basis of $\TriDend$.

We proceed with exploring the whole rewriting system we presented. First of all, we note that all the rewriting rules are consistent with the defining relations of $\RBAss$. Indeed, the rewriting rules \eqref{eq:TDelimR} and \eqref{eq:TDelimR'} are simply the definitions of the two tridendriform operations, each of the rewriting rules \eqref{eq:TD-a-R} and \eqref{eq:TD-R-b} is in fact equivalent to the Rota-Baxter relation \eqref{eq:RB}, and the remaining relations, as we already mentioned, are defining relations of the tridendriform operad which are known to hold in $\RBAss$. Second, this rewriting system is terminating. For this, we note that the rules \eqref{eq:TDelimR}--\eqref{eq:TD-R-b} move the operation $R$ further from the leaves of a tree monomial, so they can only be applied finitely many times. After that, we are left with the rewriting rules that do not involve $R$, and we demonstrated them to come from an admissible ordering of monomials, so termination is automatic. Finally, let us show that this rewriting system is confluent. Confluence of rewriting rules not involving $R$ follows from the fact that they form a Gr\"obner basis of $\TriDend$. Confluence of rewriting rules \eqref{eq:TDelimR}--\eqref{eq:TDelimR'} is immediate using the rules \eqref{eq:TD-a-R} and \eqref{eq:TD-R-b}. Finally, confluence between the rules invoving $R$ and the rules not involving $R$ is a simple calculation that we omit here. 

We observe that the left hand sides of our rewriting rules do not involve tree monomials obtained by the right module action of $\TriDend$, so \cite[Th.~4(2)]{Dot09} applies, proving freeness as a right module.  
\end{proof}

To conclude this section, let us make a remark on non-existence of PBW type results in one similar case. In \cite{Gu3}, it is established that a PBW type theorem does not hold for universal enveloping Rota--Baxter associative algebras (of weight zero) of dendriform algebras. From the categorical viewpoint, this is very easy to see. Indeed, Formula \eqref{eq:rmod} also holds in this case; both the left hand side and the right hand side are equal to $a_1\cdot R(a_2)\cdot a_3$. However, in the dendriform case, the operation $a_1\cdot a_2$ is not a part of the structure, so the coset of this operation represents a nontrivial right module generator, so Formula \eqref{eq:rmod} represents a nontrivial relation in the corresponding module. That said, for non-existence of PBW theorems, the result of \cite{Gu3} is stronger: it shows that even a non-canonical PBW type result (where isomorphisms are not fully functorial with respect to algebra morphisms) cannot hold.

\section{PBW theorem for universal enveloping Rota-Baxter Lie algebras of post-Lie algebras}

In this section, we prove the PBW theorem for universal enveloping $\RBLie$-algebras of $\PostLie$-algebras. The argument is similar to that of Theorem \ref{th:TDRB}, but uses shuffle operads. 

\begin{theorem}\label{th:PLRB}
The datum $(\PostLie,\RBLie,\alpha)$ has the PBW property. 
\end{theorem}

\begin{proof}
It is enough to prove that the shuffle operad $\RBLie$ is free as a right $\PostLie$-module. As above, to not complicate the notation, we identify operations with their images under $\alpha$, so we use the name $a_1\triangleleft a_2$ for the operation $[a_1, R(a_2)]\in \RBLie$. In line with the usual shuffle operad approach, we denote the opposite operation $a_2\triangleleft a_1$ by $a_1\overline{\triangleleft} a_2$.

Let us show that the shuffle operad $\RBLie$ can be presented by the following rewriting system:
\begin{gather}
[a_1, R(a_2)]\to a_1\triangleleft a_2,\label{eq:PLelimR}\\
[R(a_1),a_2]\to -a_1\overline{\triangleleft} a_2,\label{eq:PLelimR'}\\
R(a_1)\triangleleft a_2\to R(a_1\triangleleft a_2-a_1\overline{\triangleleft} a_2+[a_1,a_2]),\label{eq:PL-R-a}\\
a_1\overline{\triangleleft} R(a_2)\to R(-a_1\triangleleft a_2+a_1\overline{\triangleleft} a_2-[a_1,a_2]),\label{eq:PL-a-R}\\
a_1\overline{\triangleleft}[a_2,a_3]\to  - [a_1\overline{\triangleleft} a_3 ,a_2]  +  [a_1\overline{\triangleleft} a_2,a_3],\\
[a_1,a_2\overline{\triangleleft}a_3] \to  [a_1,a_3]\triangleleft a_2  -  [a_1 \triangleleft a_2,a_3],\\
[a_1,a_2\triangleleft a_3]\to [a_1,a_2]\triangleleft a_3 - [a_1\triangleleft a_3,a_2],\\
[a_1,[a_2,a_3]] \to -[[a_1,a_3],a_2] + [[a_1,a_2],a_3],\\
a_1\overline{\triangleleft} (a_2\triangleleft a_3)\to (a_1\overline{\triangleleft} a_2)\triangleleft a_3-(a_1\triangleleft a_3)\overline{\triangleleft} a_2 + (a_1\overline{\triangleleft} a_3)\overline{\triangleleft} a_2 - [a_1,a_3]\overline{\triangleleft} a_2,\\
a_1\overline{\triangleleft}(a_2\overline{\triangleleft}a_3)\to (a_1\overline{\triangleleft}a_3)\triangleleft a_2 - (a_1\triangleleft a_2)\overline{\triangleleft}a_3 + (a_1\overline{\triangleleft}a_2)\overline{\triangleleft}a_3 - [a_1,a_2]\overline{\triangleleft}a_3,\\
a_1\triangleleft (a_2\triangleleft a_3)\to a_1\triangleleft(a_2\overline{\triangleleft}a_3)-a_1\triangleleft[a_2,a_3] - (a_1\triangleleft a_3)\triangleleft a_2 + (a_1\triangleleft a_2)\triangleleft a_3.
\end{gather}
We first focus on the rewriting rules that do not involve $R$ directly. Those are defining relations of the post-Lie operad which are known to hold in $\RBLie$. It turns out that if we impose the reverse length path-lexicographic order with $[a_1,a_2]<(a_1\overline{\triangleleft} a_2)<(a_1\triangleleft a_2)$, these rewriting rules single out the leading terms for this admissible ordering. This means that we have an upper bound on dimensions of the Koszul dual operad $\PostLie^!$: such an upper bound is given by the number of tree monomials whose divisors are the leading terms of the relations. By an easy computation, the upper bound in arity $n$ thus obtained is $2^n-1$; this coincides with the dimension of $\PostLie^!(n)$ by inspection of free $\PostLie^!$-algebras~\cite[Theorem 3.8]{Val}. Since the bound that we obtained is sharp, this means that our relations form a Gr\"obner basis of $\PostLie$.

We proceed with exploring the whole rewriting system we presented. First of all, we note that all the rewriting rules are consistent with the defining relations of $\RBLie$. Indeed, the rewriting rules \eqref{eq:PLelimR} and \eqref{eq:PLelimR'} are simply the definitions of the two shuffle post-Lie operations, each of the rewriting rules \eqref{eq:PL-R-a} and \eqref{eq:PL-a-R} is in fact equivalent to the Rota-Baxter relation \eqref{eq:RBL}, and the remaining relations, as we already mentioned, are defining relations of the post-Lie operad which are known to hold in $\RBLie$. Second, this rewriting system is terminating. For this, we note that the rules \eqref{eq:PLelimR}--\eqref{eq:PL-a-R} move the operation $R$ further from the leaves of a tree monomial, so they can only be applied finitely many times. After that, we are left with the rewriting rules that do not involve $R$, and we demonstrated them to come from an admissible ordering of monomials, so termination is automatic. Finally, let us show that this rewriting system is confluent. Confluence of rewriting rules not involving $R$ follows from the fact that they form a Gr\"obner basis of $\PostLie$. Confluence of rewriting rules \eqref{eq:PLelimR}--\eqref{eq:PLelimR'} is immediate using the rules \eqref{eq:PL-R-a} and \eqref{eq:PL-a-R}. Finally, confluence between the rules invoving $R$ and the rules not involving $R$ is a simple calculation that we omit here. 

We observe that the left hand sides of our rewriting rules do not involve tree monomials obtained by the right module action of $\PostLie$, so \cite[Th.~4(2)]{Dot09} applies, proving freeness as a right module.  
\end{proof}

A similar argument can be used to prove a PBW type theorem for universal enveloping Rota--Baxter Lie algebras (of weight zero) of pre-Lie algebras. We leave it to the reader to modify our proof for that purpose. 

\section{PBW theorem for universal enveloping tridendriform algebras of post-Lie algebras}

In this section, we prove the PBW theorem for universal enveloping $\TriDend$-algebras of $\PostLie$-algebras. The argument is slightly more intricate, blending methods available for symmetric operads with some shuffle operad computations, making it similar to our previous work \cite{DotF19}, as well as to the pre-Lie/dendriform PBW theorem \cite[Theorem 3]{DT}. 

\begin{theorem}
The datum $(\PostLie,\TriDend,\phi)$ has the PBW property. 
\end{theorem}

\begin{proof}
It is enough to show that the operad $\TriDend$ is a free right $\PostLie$-module. As above, to not complicate the notation, we identify operations with their images under $\phi$, so we use the name $a_1\triangleleft a_2$ for the operation $a_1\prec a_2-a_2\succ a_1\in \TriDend$. 

We begin with considering different generators of the operad $\TriDend$: 
\begin{gather*}
[a_1,a_2]=a_1\cdot a_2-a_2\cdot a_1,\qquad
a_1\circ a_2=a_1\cdot a_2+a_2\cdot a_1,\\
a_1\triangleleft a_2=a_1\prec a_2-a_2\succ a_1,\qquad
a_1\triangleright a_2=a_1\prec a_2+a_2\succ a_1.
\end{gather*}

By a direct computation, all identities between these operations are consequences of the identities
\begin{gather*}
[[a_1,a_2],a_3]-[[a_1,a_3],a_2]=[a_1,[a_2,a_3]],\\ 
[a_1,a_2]\triangleleft a_3=[a_1\triangleleft a_3,a_2]+[a_1,a_2\triangleleft a_3],\\ 
\left((a_1\triangleleft a_2)\triangleleft a_3-a_1\triangleleft(a_2\triangleleft a_3)\right)-\left((a_1\triangleleft a_3)\triangleleft a_2-a_1\triangleleft(a_3\triangleleft a_2)\right)=a_1\triangleleft[a_2,a_3],\\ 
[a_1\circ a_2, a_3]=[a_1,a_3]\circ a_2+a_1\circ [a_2, a_3],\\ 
[a_1,a_2\triangleright a_3]=[a_1,a_2]\triangleright a_3+(a_1\triangleleft a_3)\circ a_2,\\ 
(a_1\circ a_2)\triangleleft a_3=(a_1\triangleleft a_3)\circ a_2+a_1\circ(a_2\triangleleft a_3),\\ 
a_1\triangleleft(a_2\triangleright a_3+a_3\triangleright a_2+a_2\circ a_3)=(a_1\triangleright a_2)\triangleleft a_3+(a_1\triangleright a_3)\triangleleft a_2,\\ 
a_1\triangleright(a_2\triangleleft a_3-a_3\triangleleft a_2-[a_2,a_3])=(a_1\triangleleft a_2)\triangleright a_3-(a_1\triangleright a_3)\triangleleft a_2,\\ 
(a_1\circ a_2)\circ a_3-a_1\circ(a_2\circ a_3)=[[a_1,a_3],a_2],\\ 
(a_1\triangleright a_2)\triangleright a_3 + (a_1\triangleleft a_3)\triangleleft a_2 = a_1\triangleright (a_2\triangleright a_3 + a_3\triangleright a_2 + a_2\circ a_3), \\ 
(a_1\circ a_2)\triangleright a_3 + [a_1,a_2]\triangleleft a_3=a_1\circ (a_2\triangleright a_3)+[a_1,a_2\triangleleft a_3]. 
\end{gather*}

Let us consider the filtration of the operad $\TriDend$ by powers of the two-sided ideal generated by the image of $\phi$, i.e. the ideal generated by the operations $[a_1,a_2]$ and $a_1\triangleleft a_2$. In the associated graded operad, the identities we determined become
\begin{gather}
[[a_1,a_2],a_3]-[[a_1,a_3],a_2]=[a_1,[a_2,a_3]],\label{eq:PostLie1}\\ 
[a_1,a_2]\triangleleft a_3=[a_1\triangleleft a_3,a_2]+[a_1,a_2\triangleleft a_3],\label{eq:PostLie2}\\ 
\left((a_1\triangleleft a_2)\triangleleft a_3-a_1\triangleleft(a_2\triangleleft a_3)\right)-\left((a_1\triangleleft a_3)\triangleleft a_2-a_1\triangleleft(a_3\triangleleft a_2)\right)=a_1\triangleleft[a_2,a_3],\label{eq:PostLie3}\\ 
[a_1\circ a_2, a_3]=[a_1,a_3]\circ a_2+a_1\circ [a_2, a_3],\label{eq:distr1}\\ 
[a_1,a_2\triangleright a_3]=[a_1,a_2]\triangleright a_3+(a_1\triangleleft a_3)\circ a_2,\label{eq:distr2}\\ 
(a_1\circ a_2)\triangleleft a_3=(a_1\triangleleft a_3)\circ a_2+a_1\circ(a_2\triangleleft a_3),\label{eq:distr3}\\ 
a_1\triangleleft(a_2\triangleright a_3+a_3\triangleright a_2+a_2\circ a_3)=(a_1\triangleright a_2)\triangleleft a_3+(a_1\triangleright a_3)\triangleleft a_2,\label{eq:distr4}\\ 
a_1\triangleright(a_2\triangleleft a_3-a_3\triangleleft a_2-[a_2,a_3])=(a_1\triangleleft a_2)\triangleright a_3-(a_1\triangleright a_3)\triangleleft a_2,\label{eq:distr5}\\ 
(a_1\circ a_2)\circ a_3=a_1\circ(a_2\circ a_3),\label{eq:ComTriAs1}\\ 
(a_1\triangleright a_2)\triangleright a_3  = a_1\triangleright (a_2\triangleright a_3 + a_3\triangleright a_2 + a_2\circ a_3),\label{eq:ComTriAs2} \\ 
(a_1\circ a_2)\triangleright a_3 =a_1\circ (a_2\triangleright a_3). \label{eq:ComTriAs3}
\end{gather}

These are precisely the defining relations of the operad $\PostPoisson$ controlling (right) post-Poisson algebras, see \cite[Section A.4]{BBGN} and \cite{NB}. Thus, there is a surjective map of operads 
 \[
\PostPoisson\twoheadrightarrow\gr\TriDend .
 \]
The operad $\PostPoisson$ has a suboperad generated by the operations $\circ$ and $\triangleright$ with defining relations \eqref{eq:ComTriAs1}--\eqref{eq:ComTriAs3}; this operad is usually denoted $\ComTriAss$ and controls the so called commutative triassociative algebras. In addition to those relations, the relations of the operad $\PostPoisson$ include the defining relations of the operad $\PostLie$ (identities \eqref{eq:PostLie1}--\eqref{eq:PostLie3}) and the compatibility relations, of which \eqref{eq:distr1} and \eqref{eq:distr2} say that the adjoint action of $[-,-]$ on $\ComTriAss$ is by derivations, \eqref{eq:distr3} and \eqref{eq:distr5} say that the right action of $(-\triangleleft-)$ on $\ComTriAss$ can be rewritten as a combination of tree monomials whose root label is in $\ComTriAss$, and finally \eqref{eq:distr4} (in conjunction with \eqref{eq:distr5}) say that some of the left actions of $(-\triangleleft-)$ on $\ComTriAss$ can be rewritten as a combination of tree monomials whose root label is in $\ComTriAss$. This can be used to devise a terminating rewriting system, where monomials that are being rewritten are precisely the actions on $\ComTriAss$ we just described. This shows that on the level of the underlying nonsymmetric sequences, there is a surjection onto the shuffle operad $\PostPoisson$ from the shuffle operad associated to the symmetric operad with defining relations
\begin{gather*}
[[a_1,a_2],a_3]-[[a_1,a_3],a_2]=[a_1,[a_2,a_3]],\\ 
[a_1,a_2]\triangleleft a_3=[a_1\triangleleft a_3,a_2]+[a_1,a_2\triangleleft a_3],\\ 
\left((a_1\triangleleft a_2)\triangleleft a_3-a_1\triangleleft(a_2\triangleleft a_3)\right)-\left((a_1\triangleleft a_3)\triangleleft a_2-a_1\triangleleft(a_3\triangleleft a_2)\right)=a_1\triangleleft[a_2,a_3],\\ 
[a_1\circ a_2, a_3]=0,\qquad
[a_1,a_2\triangleright a_3]=0,\qquad
(a_1\circ a_2)\triangleleft a_3=0,\\ 
a_1\triangleleft(a_2\triangleright a_3+a_3\triangleright a_2+a_2\circ a_3)=0,\\ 
(a_1\triangleright a_3)\triangleleft a_2=0,\qquad
(a_1\circ a_2)\circ a_3=a_1\circ(a_2\circ a_3),\\ 
(a_1\triangleright a_2)\triangleright a_3  = a_1\triangleright (a_2\triangleright a_3 + a_3\triangleright a_2 + a_2\circ a_3), \\ 
(a_1\circ a_2)\triangleright a_3 =a_1\circ (a_2\triangleright a_3). 
\end{gather*}
This shuffle operad is generated by the six elements $[-,-]$, $(-\circ-)$, $(-\triangleleft-)$, $(-\triangleright-)$, $(-\bar{\triangleleft}-)$, $(-\bar{\triangleright}-)$, where the last two operations are, as always, the opposites of the corresponding operations. We consider the path-lexicographic ordering of tree monomials induced by the ordering 
 \[
[-,-]< (-\bar{\triangleleft}-) < (-\triangleleft-) < (-\circ-) < (-\triangleright-) < (-\bar{\triangleright}-) .
 \]
A slightly tedious but direct computation shows that the operad above has a quadratic Gr\"obner basis for this ordering of monomials; moreover, the space of arity four elements of this operad has dimension $1080$. However, the dimension of $\TriDend(4)$ is equal to $45\cdot 4!=1080$ as well, so both surjections constructed along the way are isomorphisms, and our rewriting system for the shuffle operad $\PostPoisson$ is convergent. 
Finally, we observe that the leading terms of our Gr\"obner basis do not involve tree monomials obtained by the right module action of $\PostLie$, so \cite[Th.~4(2)]{Dot09} applies, proving freeness of $\PostPoisson$ as a right module; by a spectral sequence argument the same applies to $\TriDend$. 
\end{proof}

We remark that as in \cite{DT}, our proof implies that the operad $\PostPoisson$ is Koszul; to the best of our knowledge, this result is new.

\bibliographystyle{amsplain}
\providecommand{\bysame}{\leavevmode\hbox to3em{\hrulefill}\thinspace}

\end{document}